\documentclass[12pt, a4paper]{amsart}
\usepackage{amsfonts, amssymb, amsmath}
\usepackage{amsthm}
\usepackage[margin=2.5cm]{geometry}
\usepackage{enumerate}
\usepackage[all]{xy}

\newcommand{\N}{{\mathbb N}}

\newcommand{\Q}{{\mathbb Q}}

\newcommand{\Pp}{{\mathbb P}}
\newcommand{\A}{{\mathbb A}}
\newcommand{\Cc}{{\mathbb C}}
\newcommand{\Oo}{\mathcal{O}}

\newcommand{\se}[2]{\left\lbrace #1 \mbox{ }\vline\mbox{ } #2 \right\rbrace}
\newcommand{\ctop}[1]{\chi_{top}(#1)}

\newcommand{\tl}[1]{\tilde{#1}}

\newcommand{\sep}[2]{\left\lbrace\begin{matrix} #1 & \mbox{ if } k=1, \\ #2 & \mbox{ if } k>1. \end{matrix}\right.}

\newtheorem{thm}{Theorem}[section]
\newtheorem{pro}[thm]{Proposition}
\newtheorem{cor}[thm]{Corollary}
\newtheorem{lem}[thm]{Lemma}
\theoremstyle{definition}
\newtheorem{rk}[thm]{Remark}
\newtheorem{eg}[thm]{Example}
\newtheorem*{cov}{Convention}
\newtheorem*{defn}{Definition}
\newtheorem*{cl}{Claim}

\begin{document}
\title{Betti numbers in three dimensional minimal model program}
\author{Hsin-Ku Chen}
\address{Department of Mathematics, National Taiwan University, No. 1, Sec. 4, Roosevelt Rd., Taipei 10617, Taiwan} 
\email{d02221002@ntu.edu.tw}
\date{}
\begin{abstract}
	Let $X$ be a smooth projective threefold. We prove that, among the process of minimal model program, the variance of the third Betti number can be bounded by some integer
	depends only on the Picard number of $X$.
\end{abstract}
\maketitle

\section{Introduction}
Minimal model program plays an important role in birational geometry. Given a smooth projective threefold, the minimal model program produces a finite sequence of birational maps,
including divisorial contractions and flips. The final object
is either a minimal model, that is, a projective threefold with at most terminal singularity and nef canonical divisor, or a Mori fiber space,
which is a fibration with relatively ample canonical divisor. It is a natural and interesting question to compare the
original variety and its minimal model. Thanks to the recent attempt of understanding three dimensional minimal model program, an explicit description of elementary birational
maps between threefolds is known. With these works, it is thus possible to compare invariants between biratioanl equivalent models of threefolds.\par
In this paper we compare the change of Betti numbers under the process of minimal model program. Betti numbers are important topological invariants of algebraic varieties
and it may be used to bound some geometrical invariants. For example, in \cite{CT}, Cascini and Tasin use Betti numbers to bound $K_Y^3-K_X^3$ where $Y\rightarrow X$ is a
step of minimal model program begin with a smooth threefold $X_0$. If $Y\rightarrow X$ is a divisorial contraction to point, then $K_Y^3-K_X^3$ can be bounded by $2^{10}b_2(X_0)$.
If $Y\rightarrow X$ is blow-up smooth curve, then $K_Y^3-K_X^3$ can be bounded by some constant depends only on $b_3(Y)$ and the cubic form of $Y$.\par
In fact we are motivated by a lemma in \cite{CT}.
\begin{lem}[\cite{CT}, Lemma 2.17]
 	Let $Y\rightarrow X$ be an elementary divisorial contraction within $\Q$-factorial projective threefolds with terminal singularities.
 	Then $b_i(Y)=b_i(X)$ if $i=0,1,5,6$, and $b_i(Y)=b_i(X)+1$ if $i=2,4$.
\end{lem}
Under divisorial contractions all the Betti numbers vary regularly except for $b_3$. A natural question is: how does $b_3$ change, and is there the same phenomenon for flips?
In this article we answer this question.
\begin{thm}\label{mthm}
	Let $X$ be a smooth threefold and $X=X_0\dashrightarrow X_1\dashrightarrow...\dashrightarrow X_m=X_{min}$ be a process of minimal model program. Then
	\begin{enumerate}[(i)]
	\item $b_i(X_j)=b_i(X)$ for $i=0,1,5,6$ and for all $j$.
	\item If $j>k$, then $b_i(X_j)\leq b_i(X_k)$ for $i=2,4$. Equality holds if and only if $X_j$ and $X_k$ are connected by flips.
	\item There exists an integer $\bar{\Phi}_{\rho(X)}$ depends only on the Picard number of $X$, such that $b_3(X_j)\leq\bar{\Phi}_{\rho(X)}+b_3(X)$ for all $j$.
	\end{enumerate}
\end{thm}

We first deal with the divisorial contraction case. Since all other Betti numbers are known, to look at the variance of $b_3$ is equivalence
to look at the variance of the topological Euler characteristic and, to find the change of topological Euler characteristic is equivalent to find the
topological Euler characteristic of the exceptional
divisor. Thanks for the classification of extremal divisorial contractions to points due to Hayakawa, Kawakita and Yamamoto, cf. \cite{Hay1}, \cite{Hay2}, \cite{Hay3}, \cite{Hay4}, \cite{Hay5},
\cite{Kaw1}, \cite{Kaw2}, \cite{Kaw3} and \cite{Yam}, every extremal divisorial
contraction to point can be viewed as a weighted blow-up of LCI locus in a cyclic quotient of $\A^4$ or $\A^5$ and hence the exceptional divisor will be a LCI locus in some weighted projective
spaces. So the first step to solve our problem is to estimate the topological Euler characteristic of varieties in weighted projective spaces.\par
The main technical ingredient of our work is the following.
\begin{thm}\label{thm} Fix three positive integers $n$, $k$ and $d$. 
\begin{enumerate}[(i)]
\item There is an integer $N^n_{d,k}$ such that for any algebraic set
	$X_I\subset\A^n$ defined by an ideal $I=(f_1,...,f_k)$ with $\deg f_i\leq d$ for all $i$,
	we have $|\ctop{X_I}|\leq N^n_{d,k}$.
\item There is an integer $M^n_{d,k}$ such that for any zero locus
	$Y_I\subset\Pp(a_0,...,a_n)$ defined by an weighted homogeneous ideal $I=(f_1,...,f_k)$
	with $wt(f_i)\leq d$ for all $i$, we have
	$|\ctop{Y_I}|\leq M^n_{d,k}$, for arbitrary integer $a_i$, $i=0,...,n$.
\end{enumerate}
\end{thm}
With this theorem, one could estimate the change of $b_3$ after divisorial contractions to point. To go further, we use the factorization in \cite{CH}, which
factorizes any extremal birational maps into composition of divisorial contractions to point, blow-up LCI curves, flops and the inverse of maps above. The Betti numbers won't
change after flop, and the change of $b_3$ after blowing-up LCI curve can be easily computed. So the problem could be solved.

As an application of Theorem \ref{mthm}, we try to bound the intersection Betti numbers. Intersection homology was developed by Mark Goresky and Robert MacPherson in 1970's,
which is defined on singular manifolds and satisfied some nice properties as original singular homology on smooth manifolds.
One may expect that the difference of original Betti number and the intersection Betti number can be controlled by the singularity.
In this paper we prove a weaker statement. We will denote by $IH^i(X,\Q)$ the middle-perversity intersection cohomology group and let $Ib_i(X)$ be the dimension of $IH^i(X,\Q)$.
\begin{thm}\label{iBetti}
	Let $X$ be a projective $\Q$-factorial terminal threefold over $\Cc$. Then there is an integer $\Theta_i$ depends only on the singularity of $X$ and the Picard number $\rho(X)$,
	such that \[Ib_i(X)\leq b_i(X)+\Theta_i.\]
\end{thm}
The idea is to compare the both two kinds of Betti numbers with a smooth model. On smooth model the two Betti numbers coincide. Thanks to the main theorem of \cite{C},
there is a smooth variety $Y$ such that $Y\rightarrow X$ is a composition of extremal divisorial contractions, so Theorem \ref{mthm} applies. As stated in \cite{CT}
the intersection Betti numbers always decrease under divisorial contractions, hence there is no difficulty to derive Theorem \ref{iBetti}.

There are some question remained. In the prove of Theorem \ref{iBetti} the Picard number plays an essential role. However in the topological viewpoint
it seems no reason that this term should appear. Can one find another bound which is only depends on the singularities?
The other problem is that the converse of the inequality are much interesting. Can one prove that $b_i(X)\leq Ib_i(X)+\Theta'_i$ for some $\Theta'_i$ depends only on singularities?
Since intersection Betti numbers plays the same role on singular varieties as original Betti numbers on smooth varieties, one may regard the intersection Betti number as an elementary
quantum on singular varieties. Then, use this quantum to bound the original Betti number seems more natural.

This article is structured as follows.
In Section \ref{spre} we will quickly review some facts about terminal threefolds, involving the factorization of \cite{CH}, the discussion of invariants in terminal
threefolds and the computation of Betti numbers for some easy cases. We will prove Theorem \ref{thm} in Section \ref{stop} and Theorem \ref{mthm} in Section \ref{sBetti}.
The last section consists some examples and the proof of Theorem \ref{iBetti}.

I want to thank my advisor, Jungkai Alfred Chen. With whom the original idea of the proof were illuminated, and without his advise I can not finish this article.
Also I would like to thank Paolo Casnini for his kindly and helpful comments, particularly about the application on intersection Betti numbers. 

\section{Preliminaries}\label{spre}
\subsection{Geometry of terminal threefolds}
We define several invariants of terminal threefold, which is useful when studying threefold geometry.
\begin{defn}
	Let $X$ be a terminal threefold. A \emph{$w$-morphism} is a extremal divisorial contraction which contract exceptional divisor to a point of index $r>1$, such that
	the discrepancy of the exceptional divisor is $1/r$.\par
	The depth of $X$, denoted by $dep(X)$, is the minimal length of sequence of $w$-morphisms $X_n\rightarrow X_{n-1}\rightarrow ...\rightarrow X_1\rightarrow X$,
	such that $X_n$ is Gorenstein.
	Note that by \cite{Hay2} Theorem 1.2, for any terminal threefold $X$, $dep(X)$ exists and is finite.
\end{defn}
\begin{pro}[\cite{CH}, Proposition 2.15]\label{depdiv}
	If $f:Y\supset E\rightarrow X\ni P$ be (the germ of) a divisorial contraction to a point. Then $dep(Y)\geq dep(X)-1$.
\end{pro} 
\begin{pro}[\cite{CH}, Proposition 3.8]\label{depflip}
	Let $X\rightarrow W$ be a flipping contraction and $X\dashrightarrow X'$ be the flip, then $dep(X)>dep(X')$.
\end{pro}
\begin{rk}\label{Gcurve}
	Let $X$ be a terminal threefold. Then $dep(X)=0$ if and only if $X$ is Gorenstein. In this case, by Corollary 0.1 of \cite{B}, there is no flipping contraction.
	Also, if $X\rightarrow W$ is a divisorial contraction to curve, then $X$ is obtained by blowing up a LCI curve on $W$ (cf. \cite{Cut}, Theorem 4). 
\end{rk}
\begin{defn}
	Let $(X,P)$ be a germ of terminal threefold. It is known (cf. \cite{Mo}, Proposition 1b.3) that the singular point $P$ can be deformed into cyclic quotient points
	$P_1$, ..., $P_k$. The number $k$ is called the \emph{axial weight} of $(X,P)$ and will be denoted by $aw(P\in X)$. One can define $\Xi(P\in X)=\sum_{i=1}^k index(P_i)$.
	We will write $aw(X)=\sum_{P\in Sing(X)}aw(P\in X)$ and $\Xi(X)=\sum_{P\in Sing(X)}\Xi(P\in X)$. It is obvious that $aw(X)<\Xi(X)$.
\end{defn}
\begin{lem}[\cite{CZ}, Lemma 3.2]\label{xidep}
	Let $X$ be a terminal projective variety of dimension $3$, then $\Xi(X)\leq2dep(X)$.
\end{lem}
\begin{pro}[\cite{CZ}, Proposition 3.3]
	Let $X$ be a smooth projective threefold and assume that \[X=X_0\dashrightarrow X_1\dashrightarrow ...\dashrightarrow X_k=Z\]
	is a sequence of steps for the $K_X$-minimal model program of $X$. Then $\Xi(Z)\leq2\rho(X)$.
\end{pro}
\begin{rk}\label{deprho}
	In the proof of \cite{CZ} Proposition 3.3, one can see that $dep(Z)\leq\rho(X)$. We will use this result later.
\end{rk}

The next important result is the factorization in \cite{CH}.
\begin{thm}[\cite{CH}, Theorem 3.3]\label{ch}
	Let $g:X\subset C\rightarrow W\ni P$ be an extremal neighborhood which is isolated (resp. divisorial). If $X$ is not Gorenstein, then we have a diagram
	\[\xymatrix{ Y \ar@{-->}[rr]\ar[d]^f & & Y' \ar[d]_{f'} \\ X\ar[rd]^g && X'\ar[ld]_{g'}	\\ & W &}\]
	where $Y\dashrightarrow Y'$ consists of flips and flops over $W$, $f$ is a $w$-morphism, $f'$ is a divisorial contraction (resp. a divisorial contraction to a curve) and $g':X'\rightarrow W$
	is the flip of $g$ (resp. $g'$ is divisorial contraction to a point).
\end{thm}
\begin{rk}\label{rch}
	The diagram above satisfied more properties. 
	\begin{enumerate}[(i)]
	\item $dep(Y)=dep(X)-1$. This is by the construction of $Y$ in \cite{CH}.
	\item Assume that $Y\dashrightarrow Y'$ is decomposed into $Y=Y_0\dashrightarrow Y_1\dashrightarrow...\dashrightarrow Y_l=Y'$, then $Y_i\dashrightarrow Y_{i+1}$ is a flip for $i>0$.
		This is the step 4 in the proof of Theorem 3.3 in \cite{CH}.
	\end{enumerate}
\end{rk}
\subsection{Topology of terminal threefolds}
We will compute the change of Betti numbers under threefold birational maps in this subsection. All Betti numbers are known except for $b_3$.
\begin{lem}[\cite{CT}, Lemma 2.17]\label{BettiD}
 	Let $Y\rightarrow X$ be an elementary divisorial contraction within $\Q$-factorial projective threefolds with terminal singularities.
 	Then $b_i(Y)=b_i(X)$ if $i=0,1,5,6$, and $b_i(Y)=b_i(X)+1$ if $i=2,4$.
 \end{lem}
\begin{cor}\label{ctop}
	If $X\rightarrow W$ is extremal divisorial contraction, then \[b_3(W)-b_3(X)=\ctop{X}-\ctop{W}-2.\]
\end{cor}
\begin{pro}\label{b2}
	 Let $X$ be a smooth three-fold and $X=X_0\dashrightarrow X_1\dashrightarrow...\dashrightarrow X_m=X_{min}$ is the process of minimal model program. Then $b_0$, $b_1$, $b_5$ and $b_6$
	 are constant among the sequence $\{X_i\}$ and both $b_2$ and $b_4$ are decreasing. Moreover, $b_2$ and $b_4$ are strictly decease by one if $X_i\rightarrow X_{i+1}$
	 is a divisorial contraction, and remain unchange if $X_i\dashrightarrow X_{i+1}$ is a flip.
\end{pro}
\begin{proof} Proposition \ref{BettiD} asserts the divisorial contraction case. Assume that $X_i\dashrightarrow X_{i+1}$ is a flip. We will apply Theorem \ref{ch}
	and induction on $dep(X_i)$. One has the diagram
	\[\xymatrix{ Y \ar@{-->}[rr]\ar[d]^f & & Y' \ar[d]_{f'} \\ X_i\ar[rd]^g && X'\ar[ld]_{g'}	\\ & W &}.\]
	Note that by Remark \ref{rch} we have $dep(Y)=dep(X)-1$. One can write \[Y=Y_0\dashrightarrow Y_1\dashrightarrow...\dashrightarrow Y_l=Y'\] and
	$Y_j\dashrightarrow Y_{j+1}$ is a flip or flop for all $j$, hence $dep(Y_j)\geq dep(Y_{j+1})$ by Proposition \ref{depflip}. By induction hypothesis and the fact that Betti numbers in invariant
	after flop,	we have \[b_i(Y)=b_i(Y')\mbox{ for }i\neq3.\] Hence \[b_i(X)=b_i(Y)=b_i(Y')=b_i(X')\mbox{ for }i=0,1,5,6\] and
	\[b_i(X)=b_i(Y)-1=b_i(Y')-1=b_i(X')\mbox{ for }i=2,4.\]
\end{proof}

\section{The estimate on topology}\label{stop}
The purpose of this section is to prove Theorem \ref{thm}.
We shall prove that there is an integer $N^n_{d,k}$ such that any algebraic set in $\A^n$ defined by $k$ polynomials of degree $\leq d$ has topological Euler characteristic bounded by $N^n_{d,k}$.
Similarly, there is an integer $M^n_{d,k}$ such that an algebraic set in a weighted projective space of dimension $m$ which is defined by $k$ weighted homogeneous polynomials of weight $\leq d$
has topological Euler characteristic bounded by this integer.
To prove the existence of such kind of integers, the basic idea is to reduce the question into lower dimensional cases. We will prove:
\begin{pro}\label{Ncase}
	Assume that $N^{n-1}_{e,l}$ exists for all $e$, $l\in\N$,
	then $N^n_{d,k}$ exists.
\end{pro}
\begin{pro}\label{Mcase} If $M^{m}_{d,k}$ exists for all $m<n$ and $N^n_{d,l}$ exists for all $l\geq k$,
	then $M^n_{d,k}$ exists.
\end{pro}
\subsection{The existence of $N$-constant}

In this subsection we prove Proposition \ref{Ncase}. 
Given $X_I\subset\A^n$, where $I=(f_1,...,f_k)$ satisfying $\deg f_i\leq d$. Consider the natural map
\[ K[x_1,...,x_{n-1}]\hookrightarrow K[x_1,...,x_n]\rightarrow K[x_1,...,x_n]/I,\]
here $K$ is the ground field. This gives a morphism $\phi$ from $X_I$ to $\{x_n=0\}\cong\A^{n-1}$.
Fix $p=(a_1,...,a_{n-1})\in \A^{n-1}$, then \[\phi^{-1}(p)=\se{(a_1,...,a_{n-1},x_n)\in\A^n}{f_1(a_1,...,a_{n-1},x_n)=...=f_k(a_1,...,a_{n-1},x_n)=0}.\]
Thus $\phi^{-1}(p)$ can be studied via the equations $f_1$, ..., $f_k$. Now assume that the topology of the image is known, then since the fibres can be studied,
the topology of the original space $X_I$ could be computed. This is the reason that one can reduce the problem to the lower dimensional case.\par
For the induction reason, we will prove a stronger statement.

\begin{pro}\label{NN}
	Assume $N^{n-1}_{c,m}$ exists for all integers $c$ and $m$. Let $Z$ be an algebraic subset in $\A^{n-1}$ which is defined by an ideal
	$J=(g_1,...,g_l)$ and assuming $\deg g_j\leq e$ for some constant $e$.
	Then there is an fixed integer $L^n_{d,k,e,l}$ such that $|\ctop{\phi^{-1}Z}|\leq L^n_{d,k,e,l}$.
\end{pro}

We divide this subsection into four parts. In the first part we study the common roots of a collection of polynomials, which is the main tool
we will use to study the fibre of the projection $\phi$. After the tool is developed, 
we could get much information between the points in $H$ and its fibre in $\A^n$, provided the degree of $f_1$, ... $f_k$ do not be too small.
This is the second part of this subsection. In the third part we deal with the case when the degree of $f_i$ is too small for some $i$ so that above technique does not work.
Finally in the last part we run a complicated induction and prove Proposition \ref{NN}.

\subsubsection{The generalized resultant}
We generalize the idea of the \emph{resultant} in classical algebra to
describe the condition that a collection of polynomials has a common zero.\par
Let $g_1$, ..., $g_k\in K[x]$ be one variable polynomials with $\deg g_i=d_i>0$.
One write $g_i=\sum_ja_{i,j}x^j$ and we will denote
\[A^i_{g_1,...,g_k}=\begin{pmatrix}a_{i,d_i} & & & \\ a_{i,d_i-1} & a_{i,d_i} &  & \\
	\vdots & \vdots & \ddots & a_{i,d_i} \\ a_{i,0} & \vdots & \ddots & \vdots \\
	& a_{i,0} &  & \vdots \\ & & & a_{i,0} \end{pmatrix} \]
which is a $(d_i+d_k)\times d_k$ matrix satisfying
	\[(A^i_{g_1,...,g_k})_{pq}=\left\lbrace
	\begin{matrix} a_{i,d_i-q+p} & 0\leq p-q\leq d_i \\ 0 & \mbox{otherwise}
	\end{matrix}\right..\]
Also define
\[B^i_{g_1,...,g_k}=\begin{pmatrix}a_{k,d_k} & & & \\ a_{k,d_k-1} & a_{k,d_k} &  & \\
	\vdots & \vdots & \ddots & a_{k,d_k} \\ a_{k,0} & \vdots & \ddots & \vdots \\
	& a_{k,0} &  & \vdots \\ & & & a_{k,0} \end{pmatrix} \]
be a $(d_i+d_k)\times d_i$ matrix such that
	\[(B^i_{g_1,...,g_k})_{pq}=\left\lbrace
	\begin{matrix} a_{k,d_k-q+p} & 0\leq p-q\leq d_k \\ 0 & \mbox{otherwise}
	\end{matrix}\right..\]
Consider
\[ T_{g_1,...,g_k}=\begin{pmatrix} A^1_{g_1,...,g_k} & B^1_{g_1,...,g_k} & 0 & \cdots & 0 \\ 
	A^2_{g_1,...,g_k} & 0 & B^2_{g_1,...,g_k} & 0 & \vdots \\ \vdots & \vdots & 0 & \ddots & 0 \\
	A^{k-1}_{g_1,...,g_k} & 0 & \cdots & 0 & B^{k-1}_{g_1,...,g_k} \end{pmatrix},\]
which is a $(d_1+...+d_{k-1}+(k-1)d_k)\times(d_1+...+d_k)$ matrix.
\begin{lem}\label{res} The polynomials $g_1$, ..., $g_k$ have common zeros 
	if and only if the matrix $T_{g_1,...,g_k}$ is not full rank. Moreover, the number of the common zeros is
	exactly the nullity of $T_{g_1,...,g_k}$, counted with multiplicity.
\end{lem}
\begin{proof}
	\begin{cl}$g_1$, ..., $g_k$ has common zero if and only if there is polynomials
	$h_1$, ..., $h_k$ such that $\deg h_i<\deg g_i$ and $h_ig_k=h_kg_i$ for all $i<k$.\end{cl}
	Indeed, if the polynomials has common zeros, then they have a common factor in the polynomial ring
	$K[x]$. So we may write $g_i=bh_i$, where $b=\gcd(g_1,...,g_k)$ and then $\deg h_i<\deg g_i$
	and $h_ig_k=h_kg_i$. Conversely, assume $h_ig_k=h_kg_i$ for some $h_1$, ..., $h_k$ with
	$\deg h_i<\deg g_i$. If $g_k$ and $h_k$ has no common root, then every root of $g_k$ is a root of
	$g_i$ for all $i$ thanks to the relation $h_ig_k=h_kg_i$. Otherwise let $l=\gcd(g_k,h_k)$ and define $\bar{g}_k=g_k/l$, $\bar{h}_k=h_k/l$.
	Then $\deg \bar{g}_k>0$. We still have the relation $h_i\bar{g}_k=\bar{h}_kg_i$ and $\gcd(\bar{g}_k,\bar{h}_k)=1$. As the previous discussion the root of $\bar{g}_k$
	will be a root of $g_i$ for all $i$.\par 
	Thus to prove the lemma, one only need to find $h_i$ satisfied the condition above. Let
	\[ v=\left( r_{k,d_k-1},...,r_{k,0},-r_{1,d_1-1},...,-r_{1,0},
		-r_{2,d_2-1},...,-r_{k-1,d_{k-1}-1},...,-r_{{k-1},0} \right)^t\]
	be a column vector in $K^{d_1+...+d_k}$, and let $h_i=\sum_jr_{i,j}x^j$, then one can check
	that the condition $h_ig_k=h_kg_i$ is exactly the linear condition $T_{g_1,...,g_k}v=0$. Hence
	$g_1$, ..., $g_k$ has common zeros if and only if $T_{g_1,...,g_k}$ is not full rank.\par
	Now notice that if $b=\gcd(g_1,...,g_k)$ and let $\alpha_i=g_i/b$, then the number of common
	zeros of $g_1$, ..., $g_k$ is exactly $\deg b$. For $1\leq j\leq\deg b$,
	Let $v_j$ be the vector in $K^{d_1+...+d_k}$ corresponds to the collection of polynomials
	$\{x^{j-1}\alpha_i\}_{i=1}^k$, then $v_j$ is lying on the null space of $M$ and $v_1$, ..., $v_{\deg b}$
	are linearly independent.\par
	Conversely assume $T_{g_1,...,g_k}w=0$ for some $w\in K^{d_1+...+d_k}$, then
	$w$ corresponds to a collection of polynomials $h_1$, ..., $h_k$ satisfying $h_ig_k=h_kg_i$
	and $\deg h_i<\deg g_i$. We claim that $\alpha_i$ divides $h_i$ for all $i$.\par
	Let $c_i=\gcd(g_i,g_k)$, $g_i=c_i\beta_i$ and $g_k=c_i\gamma_i$. The relation $h_ig_k=h_kg_i$
	yields $h_i\gamma_i=h_k\beta_i$. Since $\gcd(\beta_i,\gamma_i)=1$ we have $\gamma_i$ divides $h_k$
	for all $i$, hence $l.c.m.(\gamma_1,...,\gamma_{k-1})$ divides $h_k$. On the other hand we have the
	relation $g_k=b\alpha_k=c_i\gamma_i$. Note that $b=\gcd(c_1,...,c_{k-1})$, hence $\gamma_i$ divide
	$\alpha_k$ for all $i$ and so $l.c.m.(\gamma_1,...,\gamma_{k-1})$ divides $\alpha_k$.
	If $\alpha_k\neq l.c.m.(\gamma_1,...,\gamma_{k-1})$ then $\alpha_k/l.c.m.(\gamma_1,...,\gamma_{k-1})$
	will divide $c_i$ for all $i$, contradict to $b=\gcd(c_1,...,c_{k-1})$. Thus
	$\alpha_k=l.c.m.(\gamma_1,...,\gamma_{k-1})$ divides $h_k$. Finally the relation
	$h_ig_k=h_kg_i$ gives that $h_i\alpha_k=h_k\alpha_i$. Since $\alpha_k$ divide $h_k$, we have
	$\alpha_i$ divide $h_i$ for all $i$.\par
	Now	$\deg h_i<\deg g_i=\deg b+\deg\alpha_i$, hence $h_i=h'\alpha_i$ for some polynomial $h'$ and $\deg h'<\deg b$.
	Thus $w$ is lying on the subspace generated by $v_1$, ..., $v_{\deg b}$ and then $null(T_{g_1,...,g_k})=\deg b$ and
	the last part of the lemma is proved.
\end{proof}

\begin{lem}\label{root}
	Assume $\deg g_i>1$ for all $i$. Let \[s_0=null(T_{g_1,...,g_k})  ;\quad s_1=null(T_{g_1,...,g_k,g'_1,...,g'_k}),\]
	here $g'_i$ denotes the formal derivative of polynomials.
	Then the number of distinct common roots of $g_1$, ..., $g_k$ is exactly $s_0-s_1$.
\end{lem}
\begin{proof}
	Let $b=\gcd(g_1,...,g_k)$. We will show that $g.c.d(b,b')=\gcd(g_1,...,g_k,g'_1,...,g'_k)$. Indeed, if we write
	$g_i=bh_i$, then $g'_i=b'h_i+bh'_i$, hence $g.c.d(b,b')$ divides $g_i$ and $g'_i$ for all $i$ and then $\gcd(b,b')$ divides
	$\gcd(g_1,...,g_k,g'_1,...,g'_k)$. Conversely, if $p$ is a polynomial
	divides $g_i$ and $g'_i$ for all $i$, then $p$ will divide $\gcd(g_1,...,g_k)=b$. The condition $p$ divides $g'_i$ implies $p$ divides
	$b'h_i$ for all $i$. However, $\gcd(h_1,...,h_k)=1$. Thus $p$ divides $b'$ and hence $p$ divides $\gcd(b,b')$. That is,
	$\gcd(g_1,...,g_k,g'_1,...,g'_k)$ divides $\gcd(b,b')$.\par
	Now write $b=(x-a_1)^{r_1}...(x-a_m)^{r_m}$, then the number of distinct common roots of $g_1$, ..., $g_k$ is $m$. On the other hand,
	\[b'=\left((x-a_1)^{r_1-1}...(x-a_m)^{r_m-1}\right)\left(\sum_ir_i(x-a_1)...(x-a_{i-1})(x-a_{i+1})...(x-a_m)\right).\]
	Hence $\gcd(b,b')=(x-a_1)^{r_1-1}...(x-a_m)^{r_m-1}$. By Lemma \ref{res}, $s_0=\deg b=r_1+...+r_m$ and
	$s_1=\deg(\gcd(b,b'))=(r_1-1)+...+(r_m-1)=r_1+...+r_m-m$. A conclusion is that $s_0-s_1=m$, as we want.
\end{proof}
\subsubsection{The geometry of the projection map}
In this part we study the fibre of $\phi:X_I\rightarrow \A^{n-1}$.
We will view $f_i$ as a polynomial in $x_n$ and we will
denote $f'_i=\frac{\partial}{\partial x_n}f_i$. Let
\[ T^0=\sep{T_{f_1,f'_1}}{T_{f_1,...,f_k}}\quad\quad T^1=\sep{T_{f_1,f'_1,f''_1}}{T_{f_1,...,f_k,f'_1,...,f'_k}}\]
provided that all the polynomials are non-constant. Note that $T^0$ and $T^1$ are matrices with all entries being a polynomial in $K[x_1,...,x_{n-1}]$.\par
\begin{cov}For $j=0$, $1$, we say the \emph{condition $(A^j)$ are satisfied} if $T^j$ is defined. That is, $\deg f_i>j$ (resp. $j+1$) for all $i$ if $k>1$ (resp. $k=1$). \end{cov}
When ($A^j$) is satisfied, one could study the fiber of $\phi$ via the nullity of $T^j$. There are three possibility of the fiber of $\phi$: empty, finite points or a $\A^1$.
The fiber is a $\A^1$ at a point $P\in\A^{n-1}$ if and only if all $f_i$ vanishes at $P$, which is easy to detect. The main question is to find the locus on $\A^{n-1}$ such that
the pre-image of $\phi$ is finite, and on such locus one should find the number of points in the fiber.\par
Assume ($A^0$) one could solve the first question (cf. Lemma \ref{M0f}, Lemma \ref{M0nf}). If ($A^1$) holds and assuming more conditions one could count the cardinality
of the fiber (cf. Lemma \ref{M11}).

\begin{lem}\label{M0f}
	Assume ($A^0$). Fix $p\in \A^{n-1}$ and assume that $T^0(p)$ is
	full rank. Then \[|\phi^{-1}(p)|=\sep{\deg f_1}{0}\]
\end{lem}
\begin{proof}
	Assume $k=1$. If the leading coefficient vanishes over $p$, then the first row of $T^0$ is always zero. Since $T^0$ is a square matrix, this implies 
	$T^0$ is not full rank. Hence we may assume the leading coefficient do not vanishing at $p$, so both $f_1$ and $f'_1$ are non-constant.
	Using Lemma \ref{res}, we see that $T^0(p)$ is full rank implies $f_1$ and $f'_1$ consist no common zero. Hence $f_1$ consists no multiple roots over $p$,
	so $|\phi^{-1}(p)|=\deg f_1$.\par
	Now assume $k>1$. First assume $f_i$ is constant over $p$ for some $i$. Then if $f_i$ is identically zero, $T^0$ can not be full rank.
	On the other hand, if $f_i$ is a non-zero constant, then $\phi^{-1}(p)$ is always empty so the conclusion is always true.
	Finally assume $f_i$ is non-constant for all $i$, then for any $p\in H$, $\phi^{-1}(p)$ is non-empty only if $f_1$, ..., $f_k$ admit common zeros. By Lemma \ref{res},
	this implies the matrix $T^0$ is not full rank.
\end{proof}
\begin{lem} \label{M0nf}
	Assume ($A^0$). Given $p\in \A^{n-1}$ and assume that $T^0(p)$ is not full rank. Assume further that the leading coefficient of $f_i$
	do not vanish at $p$ for all $i$. Then if $k>1$, we have that $p$ is contained in the image of $\phi$. For $k=1$,
	one can say that $\phi$ is a finite morphism near $p$ and $p$ is lying on the ramification locus.
\end{lem}
\begin{proof}
	First assume $k>1$. The hypothesis implies that $f_1$, ..., $f_k$ is non-constant polynomial in $x_n$ over $p$. By Lemma \ref{res},
	$T^0$ is not full rank at $p$ if and only if $f_1$, ..., $f_k$ admits a common zero, say $\xi\in K$. If we write $p=(a_1,...,a_{n-1})$,
	then the point $(a_1,...,a_{n-1},\xi)$ is lying on $X_I$ and is mapped to $p$ by $\phi$. Hence $p$ is contained in the image of $\phi$.\par
	For the $k=1$ case, note that $T^0$ is defined implies $\deg f_1>1$. By assumption, the leading coefficient of $f_1$ do not vanish at $p$,
	hence it do not vanish on a neighborhood $U$ of $p$. We see that for any point $q\in U$ we have $f_1$ is a polynomial of positive degree
	in $x_n$ over $q$, so the pre-image of $\phi$ consists only finitely many points and so $\phi$ is a finite morphism on $U$. Now the
	condition that $T^0(p)$ is not full rank implies $f_1$ consists multiple root over $p$, hence $p$ is lying in the ramification locus of $\phi$.
\end{proof}

Now let $Z\subset \A^{n-1}$ be a subset contained in the image of $\phi$. For $p\in Z$ we will denote $r(p)=|\phi^{-1}(p)|$ and $r(Z)=\max_{p\in Z}\{r(p)\}$.
Also define $s_0(p)=null(T^0(p))$ and $s_1(p)=null(T^1(p))$. What we want to do is to find the locus which consists of the points
$p\in Z$ such that $r(p)\neq r(Z)$. Such point could be determined using the number $s_0$ and $s_1$, under suitable conditions.
\begin{lem}\label{M11}
	Fix $Z\subset \A^{n-1}$ be any subset. Assume that the leading coefficient of $f_i$ do not vanish over $Z$ for all $i$. When $k=1$ (resp. $k>1$) assume ($A^0$)(reps. ($A^1$)).
	Then for any $p\in Z$ we have
	\begin{enumerate}[(i)]
	\item Assume $k=1$, then $r(p)=\deg f_1(p)-s_0(p)$.
	\item Assume $k>1$, then $r(p)=s_0(p)-s_1(p)$.
	\end{enumerate}
\end{lem}
\begin{proof}
	First assume $k>1$. By Lemma \ref{root} we have $r(p)=s_0(p)-s_1(p)$ for all $p\in Z$. Now assume $k=1$.
	The assumption that $T^0$ exists and the leading coefficient of $f_1$ do not vanish implies that
	$\phi$ is a finite morphism over $Z$. For any $p$ in $Z$ the number $r(p)$ is the number of distinct roots of $f_1$ over $p$.
	Assume $f_1(p)=(x_n-a_1)^{r_1}...(x_n-a_m)^{r_m}(x_n-b_1)...(x_n-b_l)$ with $r_i>1$. We have $r(p)=m+l$,
	$\deg f_1(p)=r_1+...+r_m+l$, $s_0(p)=(r_1-1)+...+(r_m-1)=r_1+...+r_m-m$ by Lemma \ref{res}, hence $r(p)=\deg f_1-s_0(p)$.\par
\end{proof}
\begin{cor}\label{M1}
	Fix $Z\subset \A^{n-1}$. Assume that the leading coefficient of $f_i$ do not vanish over $Z$ for all $i$ and one of the following condition holds:
	\begin{enumerate}[(i)]
	\item $k=1$ and ($A^0$) holds.
	\item $k>1$, ($A^1$) holds and $s_0$ is constant over $Z$.
	\end{enumerate}	
	Then $\phi$ is a finite morphism over $Z$. When $k=1$ (resp. $k>1$) the ramification locus of $\phi$ is exactly the locus where the function $s_0$ (resp. $s_1$)
	do not reach its minimum.
\end{cor}

\subsubsection{The small degree cases}
In this section we deal with the cases that the $\deg f_i$ is too small so that ($A^0$) or ($A^1$) dose not hold.
\begin{lem}\label{nM0} 
	Under the assumption and notation in Proposition \ref{NN}, if $k=1$ and ($A^0$) dose not hold over $Z$, then the conclusion of Proposition \ref{NN} is true.
\end{lem}
\begin{proof}
	The assumption says that  $\deg f_1<2$ over $Z$. If $\deg f_1=0$, then $f_1\in K[x_1,...,x_{n-1}]$ is independent of $x_n$. Let $Z'$ be the zero locus of the ideal $J+(f_1)$,
	then $|\ctop{Z'}|\leq N^{n-1}_{\max\{e,d\},l+1}$. One see that outside $Z'$, the pre-image of $\phi$ is empty, and $\phi^{-1}Z'\cong Z'\times\A^1$.
	Hence $|\ctop{\phi^{-1}Z}|=|\ctop{\phi^{-1}Z'}|=|\ctop{Z'}|\leq N^{n-1}_{\max\{e,d\},l+1}$.\par
	On the other hand, assume $\deg f_1=1$. Write $f_1=a_1x_n+a_0$. Let $Z_0$ be the zero locus defined by $J+(a_1)$ and $Z_1=Z-Z_0$. Then $\ctop{\phi^{-1}Z_0}$ can be
	computed in the previous case since we can replace $f_1$ by $a_0$ and replace $Z$ by $Z_0$.
	On the other hand, since $f_1$ is a degree one polynomial over any points in $Z_1$, we have $\phi^{-1}Z_1\cong Z_1$.
	Now $|\ctop{\phi^{-1}Z_1}|=|\ctop{Z_1}|=|\ctop{Z}-\ctop{Z_0}|\leq N^{n-1}_{e,l}+N^{n-1}_{\max\{e,d\},l+1}$ can be compute. Thus the lemma is proved.
\end{proof}
The other case is that ($A^0$) holds but ($A^1$) dose not hold. This happened when $k=1$ and $\deg f_1=2$ or $k>1$ and $\deg f_i=1$ for some $i$.
\begin{lem}\label{nM1}
	Let $Z\subset H$ and assume the following.
	\begin{enumerate}[(i)]
		\item ($A^0$) holds but ($A^1$) dose not hold.
		\item $T^0(p)$ is not full rank for all $p\in Z$.
		\item The leading coefficient of $f_i$ do not vanishing for all $i$ for any point $p\in Z$.
	\end{enumerate}
	Then $\phi$ is one-to-one over $Z$. In particular, $\ctop{\phi^{-1}Z}=\ctop{Z}$.
\end{lem}
\begin{proof}
	First assume $k>1$. By Lemma \ref{M0nf} the assumption yields that $Z$ is contained in the image of $\phi$. On the other hand,
	$T^0$ is defined but $T^1$ is not defined implies $\deg f_i=1$ for some $i$, hence $\phi$ is one-to-one over $Z$.\par
	Now assume $k=1$. Since $T^0$ is defined but $T^1$ is not defined, we have $\deg f_1=2$, hence $\phi$ is two-to-one over some open neighborhood of $Z$.
	However, Lemma \ref{M0nf} implies that $Z$ is lying on the ramification locus, hence $\phi$ is one-to-one over $Z$.
\end{proof}

\subsubsection{The main proofs}
We will need the following lemma.
\begin{lem}\label{com}
	If $S=S_1\cup S_2\cup...\cup S_k$ for some algebraic set $S_i$. For any $I\subset\{1,...,k\}$,
	we denote $S_I=\bigcap_{i\in I}S_i$. Assume that $|\ctop{S_I}|\leq M$ for some integer $M$
	and for all $I\subset\{1,...,k\}$. Then $|\ctop{S}|\leq(2^{k}-1)M$.
\end{lem}
\begin{proof}
	We prove by induction on $k$. When $k=2$ we have
	\[\ctop{S}=\ctop{S_1-S_{12}}+\ctop{S_2-S_{12}}+\ctop{S_{12}}=\ctop{S_1}+\ctop{S_2}-\ctop{S_{12}},\]
	so $|\ctop{S}|\leq3M=(2^2-1)M$.
	In general let $S'=S_1\cup...\cup S_{k-1}$, then
	$S'\cap S_k=(S_1\cap S_k)\cup...\cup(S_{k-1}\cap S_k)$, hence $|\ctop{S'\cap S_k}|\leq(2^{k-1}-1)M$
	by induction hypothesis. We also have $|\ctop{S'}|\leq(2^{k-1}-1)M$. Thus
	\begin{align*}
		|\ctop{S}|&=|\ctop{S'-(S'\cap S_k)}+\ctop{S_k-(S'\cap S_k)}+\ctop{S'\cap S_k}|\\
		&\leq|\ctop{S'}|+|\ctop{S_k}|+|\ctop{S'\cap S_k}|\\
		&\leq (2(2^{k-1}-1)+1)M=(2^k-1)M.\end{align*}
\end{proof}
\begin{proof}[Proof of Proposition \ref{NN}.]
	We will divide $Z$ into many pieces, and treat each piece separately. In each piece,
	either the topology of the pre-image can be easily computed, or after cut out some closed subset the
	pre-image can be computed, and there is some quantum which strictly decrease after restrict to the subset
	above. In the latter case we can use induction on the special quantum and finally the problem could be solved.
	We will treat the following cases.
	\begin{itemize}
	\item[Case(I)] ($A^0$) holds.\par
	Let
	\[Z'=\se{p\in Z}{T^0(p)\mbox{ is not full rank }}\]
	and $Z''=Z-Z'$.
	By Lemma \ref{M0f}, \[\ctop{\phi^{-1}Z''}=\sep{(\deg f_1)\ctop{Z''}}{0}\]\par
	We further divide $Z'$ into \[Z_-=\se{p\in Z'}{\mbox{The leading coefficient of }f_i\mbox{ vanish over $p$ for some }i}\]
	and $Z_+=Z'-Z_-$.
	To compute $\ctop{\phi^{-1}Z_-}$, let $a_i$ be the leading coefficient of $f_i$. For $S\subset\{1,...,k\}$, let
	$W_S$ be the zero locus defined by $J+(a_{i_0}...a_{i_p})$ if $S=\{i_0,...,i_p\}$ and $J$ is the defining ideal of $Z$.
	Then $W_S$ is the locus in $Z$ such that the leading coefficient of $f_i$ vanish for all $i\in S$.
	Hence $Z_-=\bigcup_{1\leq i\leq k}W_i$ and $W_S=\bigcap_{i\in S}W_i$. Furthermore, one may induction
	on the number $\deg f_1+...+\deg f_k$ so that we may assume $\ctop{\phi^{-1}W_S}$ can be computed.
	By Lemma \ref{com}, $\ctop{\phi^{-1}Z_-}$ can be bounded.\par
	Now one has to compute $\ctop{\phi^{-1}Z_+}$. If ($A^1$) is not true, then Lemma \ref{nM1} implies that $\ctop{\phi^{-1}Z_+}=\ctop{Z_+}$. Assume ($A^1$) is true.
	We divide $Z_+$ into
	\[Z_0=\se{p\in Z_+}{s_0(p)\mbox{ reach its minimum in }Z_+}\]
	and $Z_0'=Z_+-Z_0$.
	One may replace $Z$ by $Z'_0$ and induction on $\min_{p\in Z}\{s_0(p)\}$.
	This number is increasing and always less or equal than $\deg f_i$ for all $i$, so after finite step, $Z'_0$ would be empty.\par
	If $k>1$ we further divide $Z_0$ into
	\[Z_1=\se{p\in Z_0}{s_1(p)\mbox{ reach its minimum in }Z_0}\]
	and $ Z'_1=Z_0-Z_1$.
	By Corollary \ref{M1}, when $k=1$ (resp. $k>1$) $\phi$ is unramified over $Z_0$ (resp. $Z_1$). Hence
	\[|\ctop{\phi^{-1}Z_i}|=r(Z_i)|\ctop{Z_i}|\leq d|\ctop{Z_i}|,\]
	with $i=0$ (resp. $i=1$) in $k=1$ (resp. $k>1$) case.\par
	When $k>1$ we have $r(Z'_1)<r(Z_0)$. We will replace $Z$ by $Z'_1$ and induction
	on the number $r$. When $r(Z_0)=1$ $Z'_1$ is always empty, so the induction works.
	\item[Case(II)] ($A^0$) does not hold.
		If $k=1$, this case can be solved by Lemma \ref{nM0}. Now assume $k>1$. In this case $\deg f_i=0$ for some $i$. If $f_i$ is a non-zero constant, then $\phi^{-1}Z$
		is empty, so there is nothing to prove. If $f_i$ is identically zero, we can drop out $f_i$
		from the generator of $I$, and goes to the case with smaller $k$. By induction on $k$, this situation is solved.
	\end{itemize}		
	We have to show that $Z'$, $Z_-$, $Z'_1$ and $Z'_0$ can be defined by algebraic equations, and the total number and the degree
	of those equations can be bounded by some integer depends on $d$ and $k$, so the induction could work.\par
	To see this, let $c_i$ and $r_i$ be the number of columns and rows of $T^i$, respectively, for $i=0$, $1$.
	Then $c_0\leq dk$, $r_0\leq 2(k-1)d$, and $c_1\leq 2c_0$, $r_1\leq 2r_0$. Let $R$ be the ideal containing
	all maximal minors of $T^0$, then $R$ can be generated by $C^{r_0}_{c_0}$ many generators and each generator
	is a degree at most $dr_0$ polynomial. One can see that $Z'$ is generated by $J+R$. Since $Z_-=\bigcup_{1\leq i\leq k}W_i$ and
	the defining ideal of $W_i$ are bounded, the defining ideal of $Z_-$ is bounded.\par
	
	Now let $t_i=\max_{p\in Z}{rk(T^i(p))}$.
	$p\in Z$ satisfied $s_i(p)$ do not reach minimum if and only if $s_i(p)+t_i>r_i$. Hence one only need to find those
	points in $Z$ such that the rank of $T^i$ at that point is less than $t_i$, or equivalently, all $t_i\times t_i$ minors
	of $T^i$ vanishes. Let $Q_i$ be the ideal containing all $t_i\times t_i$ minors of $T^i$, then $Q_i$ is generated by
	at most $r_ic_i$ many elements and each element is a degree at most $dt_i\leq dr_i$ polynomial in $K[x_1,...,x_{n-1}]$.
	One can easily see that $Z'_i$ is defined by $J+Q_i$ for $i=0$, $1$.\par
	The other task is to compute $\ctop{Z''}$, $\ctop{Z_+}$ and $\ctop{Z_i}$ for $i=0,1$. Since $Z'$ is generated by $J+R$, 
	\[|\ctop{Z'}|\leq N^{n-1}_{e+dr_0,l+C^{r_0}_{c_0}}.\]
	We have $\ctop{Z''}=\ctop{Z}-\ctop{Z'}$. Thus \[|\ctop{Z''}|\leq N^{n-1}_{e,l}+N^{n-1}_{e+dr_0,l+C^{r_0}_{c_0}}.\]
	Now consider $|\ctop{W_S}|\leq N^{n-1}_{e+d^{|S|},l+1}\leq N^{n-1}_{e+d^k,l+1}$
	for all $S\subset\{1,...,k\}$, hence \[|\ctop{Z_-}|\leq (2^k-1)N^{n-1}_{e+d^k,l+1}\] by Lemma \ref{com}.
	A conclusion is that $\ctop{Z_+}=\ctop{Z'}-\ctop{Z_-}$ can be bounded.\par
	Finally we try to bound $\ctop{Z_i}$. As the argument above $Z'_i$ is defined by the ideal $J+Q_i$ for $i=0$ and $1$,
	hence $|\ctop{Z'_i}|\leq N^{n-1}_{e+dr_i,l+r_ic_r}$ can be bounded. Thus
	\[\ctop{Z_0}=\ctop{Z_+}-\ctop{Z'_0}\mbox{ and }\ctop{Z_1}=\ctop{Z_0}-\ctop{Z'_1}\] can be bounded.
\end{proof}

\begin{proof}[Proof of Proposition \ref{Ncase}]
	One can take $N^n_{d,k}=L^n_{d,k,0,1}$ by considering $J$ in Proposition \ref{NN} to be the zero ideal. 
\end{proof}

\subsection{The existence of $M$-constant}

\begin{proof}[Proof of Proposition \ref{Mcase}]
	Given $Y=Y_I\subset\Pp(a_0,...,a_n)$, we may assume $Y$ dose not contained
	in $\{a_0=0\}$. Let $Y'=Y\cap\{a_0=1\}$ and $Y''=Y-Y'$, then $Y''=Y\cap\{a_0=0\}$ can be viewed
	as the zero locus of a weighted homogeneous ideal in $\Pp(a_1,...,a_n)$,
	so $|\ctop{Y''}|\leq M^{n-1}_{d,k}$.\par
	On the other hand, $Y'\subset\A^n/\frac{1}{a_0}(a_1,...,a_n)$. Let $\bar{Y}$ be the pre-image
	of $Y'$ under the natural map $\A^n\rightarrow\A^n/\frac{1}{a_0}(a_1,...,a_n)$,
	then $\bar{Y}$ is defined by an ideal $I$ generated by $k$ elements, and we may assume
	the degree of each generator of $I$ is less or equal than $d$.
	Hence $|\ctop{\bar{Y}}|\leq N^n_{d,k}$.\par
	Now $\bar{Y}\rightarrow Y'$ is a branched covering. Let $\bar{R}\subset\bar{Y}$ be the
	branched locus, and $R\subset Y'$ be the image of $\bar{R}$. One have to compute $\ctop{R}$
	and $\ctop{\bar{R}}$. Note that the morphism $A^n\rightarrow A^n/\frac{1}{a_0}(a_1,..,a_n)$
	ramified at $\{x_{i_1}=...=x_{i_l}=0\}$ for some $i_1$, ..., $i_l$. Let
	$\Xi_1$, ..., $\Xi_j$ be the irreducible commponent of the ramification locus on $\A^n$ and let
	$\bar{S_i}=\Xi_i\cap\bar{R}$. One can see that $|\ctop{\bar{S_i}}|\leq N^n_{d,k+l_i}$
	if $\Xi_i=\{x_{i_1}=...=x_{i_{l_i}}=0\}$ and
	$|\ctop{\bar{S_{i_1}}\cap...\cap\bar{S_{i_m}}}|\leq N^n_{d,k+l'}$ if
	$\Xi_{i_1}\cap...\cap\Xi_{i_m}$ is of codimension $l'$. Moreover, the number of irreducible
	components of ramification locus of $A^n\rightarrow A^n/\frac{1}{a_0}(a_1,..,a_n)$
	is less than $\max_{2\leq m\leq n}\{C^n_m\}<2^n$,
	which is a number depends only on $n$. Hence by Lemma \ref{com}, there is an integer
	$A^n_{d,k}$ depends on $n$, $d$ and $k$ such that $|\ctop{\bar{R}}|\leq A^n_{d,k}$.\par
	To find $\ctop{R}$, we denote by $S_i$ the image of $\bar{S_i}$. Consider
	$\Xi_i\cong\A^{r_i}$ for some $r_i$ and the morphism $\Xi_i\rightarrow im(\Xi_i)$ can be
	viewed as the cyclic quotient $\A^{r_i}\rightarrow\A^{r_i}/\frac{1}{m_i}(b_{i_1},...,b_{i_{r_i}})$
	for some integers $m_i$ and $b_{i_1}$, ..., $b_{i_{r_i}}$.
	Consider $S_i\subset im(\Xi_i)\cong\A^{r_i}/\frac{1}{m_i}(b_{i_1},...,b_{i_{r_i}})
		\subset\Pp(m_i,b_{i_i},...,b_{i_{r_i}})$
	and let $\tl{S_i}$ be the closure of $S_i$ in $\Pp(m_i,b_{i_i},...,b_{i_{r_i}})$.
	Now we have $|\ctop{\tl{S_i}}|\leq M^{r_i}_{d,k}$ and $|\ctop{\tl{S_i}-S_i}|\leq M^{r_i-1}_{d,k}$,
	hence $|\ctop{S_i}|\leq M^{r_i}_{d,k}+M^{r_i-1}_{d,k}$. Moreover, for any
	$i_1$, ..., $i_l\in\{1,...,j\}$, $\Xi_{i_1}\cap...\cap\Xi_{i_l}\cong\A^{r_{i_1...i_l}}$ for some
	integer $r_{i_1...i_l}$ and the same argument shows that
	$|\ctop{S_{i_1}\cap...\cap S_{i_l}}|\leq M^{r_{i_1...i_l}}_{d,k}+M^{r_{i_1...i_l}-1}_{d,k}$.
	In particular, let $M=2\max_{r<n}\{M^r_{d,k}\}$,
	then we have $|\ctop{R}|=|\ctop{S_1\cup...\cup S_j}|\leq(2^j-1)M$ by Lemma \ref{com} and $j<2^n$ as before.
	The conclusion is that there exists an integer $B^n_{d,k}$ depends only on $n$, $d$ and $k$
	such that $|\ctop{R}|\leq B^n_{d,k}$.\par
	Now we have
	\begin{align*} |\ctop{Y'}|&=|\frac{1}{a_0}(\ctop{\bar{Y}}-\ctop{\bar{R}})+\ctop{R}|\\
		& \leq|\ctop{\bar{Y}}|+|\ctop{\bar{R}}|+|\ctop{R}|\leq N^n_{d,k}+A^n_{d,k}+B^n_{d,k}.\end{align*}
	Hence $|\ctop{Y}|\leq|\ctop{Y'}|+|\ctop{Y''}|\leq M^{n-1}_{d,k}+N^n_{d,k}+A^n_{d,k}+B^n_{d,k}$ can be bounded
\end{proof}

\begin{proof}[Proof of Theorem \ref{thm}.]
	One can easily see that $M^1_{d,k}=N^1_{d,k}=d$. Hence Proposition \ref{Ncase} and Proposition \ref{Mcase} implies the theorem. 
\end{proof}
\section{The boundedness of Betti numbers}\label{sBetti}
In this sectoin we will bound the variance of $b_3$. Thanks to Corollary \ref{ctop}, it is equivalence to bound the variance of the topological Euler characteristic,
which is much easier to compute. The following statement is a corollary of Theorem \ref{thm}, which could help us to bound the variance of $\chi_{top}$ under divisorial contraction to point.
\begin{cor}\label{cor}
	Assume that $X$ is a cyclic quotient of local complete intersection locus of codimension $k$ in $\A^n$, and $Y\rightarrow X$ be a weighted blow-up of weight $\sigma$.
	If the $\sigma$-weight of the defining equation of $X$ is bounded by a constant $d$, then $|\ctop{Y}-\ctop{X}|\leq M^n_{d,k}+1$.
\end{cor}
\begin{proof}
	Write $\sigma=\frac{1}{m}(a_0,...,a_n)$. The exceptional locus $E$ of $Y\rightarrow X$ is contained in $\Pp^n(a_0,..,a_n)$ and is defined by $k$ equations with weight $\leq d$.
	Hence $|\ctop{E}|\leq M^n_{d,k}$. Now \[|\ctop{Y}-\ctop{X}|=|\ctop{E}-\ctop{pt}|\leq M^n_{d,k}+1.\]
\end{proof}
Given a divisorial contraction $Y\rightarrow X$, we will show that the difference of $\chi_{top}$ can be bound by constant depends only on $dep(X)$ in $w$-morphism case,
and on $dep(Y)$ in general cases. The reason we need the first statement is that inverse of $w$-morphism occurs in the factorization in Theorem \ref{ch}.
\begin{pro}\label{divm}
	Let $Y\rightarrow X$ be a divisorial contraction contract a divisor $E$ to a point $P\in X$. Assume the index of $P$ is $m>1$ and assume $a(E,X)=1/m$.
	Then \[|\ctop{Y}-\ctop{X}|\leq D_{dep(X)}\] for some integer $D_{dep(X)}$ depends only on the number $dep(X)$.
\end{pro}
\begin{proof}
	In \cite{Hay1} and \cite{Hay2} Hayakawa classified all extremal divisorial contraction with discrepancy $1/m$ to a higher index point. He proved that such morphism
	is always a weighted blow-up of a cyclic quotient of local complete intersection locus in $\A^4$ or $\A^5$. We will denote by $d$ the upper bound of the weight of the exceptional locus
	viewed as a subvariety in the weighted projective space. What we have to do is to show that $d$ can be determined by $dep(X)$ and then
	$|\ctop{Y}-\ctop{X}|\leq M^n_{d,k}+1$ for $(n,k)=(4,1)$ or $(5,2)$, which is an integer depends only on $dep(X)$.
	We consider the type of the germs $(X,P)$.\par
	\begin{itemize}
	\item[$cA/m$.] By \cite{Hay1}, Theorem 6.4, We have \[X\cong (xy+f(z,u)=0)\subset \A^4_{(x,y,z,u)}/\frac{1}{m}(\alpha,-\alpha,1,0).\]
		If we define $\tau$-wt$(z)=1/m$ and $\tau$-wt$(u)=1$, and assume $\tau$-wt$(f(z,u))=k$, then $Y$ is obtained by the weighted blow-up of weight
		$\frac{1}{m}(a,b,1,m)$, with $a\equiv\alpha$ mod $m$ and $a+b=mk$. Furthermore, direct computation shows that $Y$ consists two cyclic quotient singularity,
		one is of index $a$ and the other one is of index $b$. We conclude that $dep(X)\geq a+b-1$ and the exceptional locus can be viewed as a weighted hypersuface in $\Pp(a,b,1,m)$
		with weight $mk=a+b\leq dep(X)+1$, hence we take $d= dep(X)+1$.
	\item[$cAx/4$.] By \cite{Hay1}, Theorem 7.4 and Theorem 7.9, we have
		$X\cong(\phi=0)\subset \A^4_{(x,y,z,u)}/\frac{1}{4}(1,3,1,2)$ and $Y$ is the weighted blow-up with weight $v$, where $\phi$ and $v$ is one of the following.
		\begin{enumerate}
		\item $\phi=x^2+y^2+f(z,u)$, $v=\frac{1}{4}(2k+1,2k+3,1,2)$ if $k$ is even, and $\frac{1}{4}(2k+3,2k+1,1,2)$ if $k$ is odd, where $k$ is defined by
			$\tau$-wt$(f(z,u))=(2k+1)/2$, providing $\tau$-wt$(z)=1/4$ and $\tau$-wt$(u)=1/2$. In this case $Y$ consists a cyclic quotient point of index $2k+3$,
			hence $dep(X)\geq 2k+3$.
		\item $\phi=x^2\pm2xg(z,u)+y^2+h(z,u)$, $v=\frac{1}{4}(2k+5,2k+3,1,2)$, where $k$ is defined as above and $\tau$-wt$(g)=(2k+1)/4$,
			$\tau$-wt$(h)>(2k+1)/2$. $Y$ consists of a cyclic quotient point of index $2k+5$, hence $dep(X)\geq 2k+5$.
		\end{enumerate}
		In the both cases we take $d=2dep(X)-4$.
	\item[$cAx/2$.] By \cite{Hay1}, Theorem 8.4 and Theorem 8.8, we have
		$X\cong(\phi=0)\subset \A^4_{(x,y,z,u)}/\frac{1}{2}(0,1,1,1)$ and $Y$ is the weighted blow-up with weight $v$, where $\phi$ and $v$ is one of the following.
		\begin{enumerate}
		\item $\phi=x^2+y^2+f(z,u)$, $v=\frac{1}{2}(k,k+1,1,1)$ if $k$ is even, and $\frac{1}{2}(k+1,k,1,1)$ if $k$ is odd, where $k$ is defined by
			$\tau$-wt$(f(z,u))=k$, providing $\tau$-wt$(z)=1/2$ and $\tau$-wt$(u)=1/2$.
		\item $\phi=x^2\pm2xg(z,u)+y^2+h(z,u)$, $v=\frac{1}{2}(k+1,k,1,1)$, where $k$ is defined as above and $\tau$-wt$(g)=k/2$,
			$\tau$-wt$(h)>k$. 
		\end{enumerate}
		In either case $Y$ consists of a cyclic quotient point of index $k+1$, hence $dep(X)\geq k+1$ and one can take $d=2dep(X)-2$.
	\item[$cD/3$.] Use \cite{Hay1}, Theorem 9.9, Theorem 9.14, Theorem 9.20 and Theorem 9.25, we may assume
		$X\cong(\phi=0)\subset\A^4_{(x,y,z,u)}/\frac{1}{3}(2,1,1,0)$ and $Y$ is given by weighted blow-up with weight $v$, where $\phi$ and $v$ is given by one of the following.
		\begin{enumerate}
		\item $\phi=u^2+x^3+yz(y\pm z)$, $v=\frac{1}{3}(2,4,1,3)$, $\frac{1}{3}(2,1,4,3)$.
		\item $\phi=u^2+x^3+yz^2+xy^4\lambda(y^3)+y^6\mu(y^3)$, $v=\frac{1}{3}(2,4,1,3)$, $\frac{1}{3}(2,1,4,3)$.
		\item $\phi=u^2+x^3+y^3+xyz^3\alpha(z^3)+xz^4\beta(z^3)+yz^5\gamma(z^3)+z^6\delta(z^3)$, $v=\frac{1}{3}(2,4,1,3)$. 
		\item $\phi=u^2+x^3+3\xi x^2z^2+y^3+xyz^3\alpha(z^3)+xz^7\beta(z^3)+yz^8\gamma(z^3)+z^{12}\delta(z^3)$, $v=\frac{1}{3}(5,4,1,6)$.
		\end{enumerate}
		The weight only depends on the type, not on the explicit equations. One can take $d=12$.
	\item[$cE/2$.] By \cite{Hay1}, Theorem 10.11, Theorem 10.17, Theorem 10.22, Theorem 10.28, Theorem 10.33, Theorem 10.41, Theorem 10.47, Theorem 10.54,
		Theorem 10.61 and Theorem 10.67, we have $X\cong(\phi=0)\subset\A^4_{(x,y,z,u)}/\frac{1}{3}(2,1,1,0)$ and $Y$ is given by weighted blow-up with weight $v$,
		where $\phi$ and $v$ is given by one of the following.
		\begin{enumerate}
		\item $\phi=u^2+x^3+g(y,z)x+h(y,z)$, $v=\frac{1}{2}(2,3,1,3)$, $\frac{1}{2}(2,1,3,3)$.
		\item $\phi=u^2+x^3+3\xi x^2z^2+g(y,z)x+h(y,z)$, $v=\frac{1}{2}(4,3,1,5)$.
		\item $\phi=u^2+x^3+3\xi x^2z^2+g(y,z)x+h(y,z)$, $v=\frac{1}{2}(2,1,3,3)$.
		\item $\phi=u^2\pm2(\alpha xz+\beta yz^2+\gamma z^5)u+x^3+g(y,z)x+h(y,z)$, $v=\frac{1}{2}(4,3,1,7)$.
		\item $\phi=u^2+x^3+g(y,z)x+h(y,z)$, $v=\frac{1}{2}(6,5,1,9)$.
		\end{enumerate}
		Again in this case one can simply take $d=18$.
	\item[$cD/2$.] As in \cite{Hay2}, we have the following.
		\begin{enumerate}
		\item $X$ is a hyperurface in $\A^4_{(x,y,z,u)}/\frac{1}{2}(1,1,0,1)$ and $Y$ is a weighted blow-up given by one of the following.\par
			\begin{tabular}{|c|c|c|}\hline
	 		Defining equations & Blowing-up weight & Relations\\ \hline
			 & $\frac{1}{2}(1,1,2,3)$ & \\
			$u^2+xyz+x^{2a}+y^{2b}+z^c$ & $\frac{1}{2}(1,3,2,3)$ & \\ & $\frac{1}{2}(3,1,2,3)$ & \\ & $\frac{1}{2}(1,1,4,3)$ & \\ \hline
			 & $\frac{1}{2}(1,l,4,l+2)$ & $l\leq aw(X)+2$ \\
			& $\frac{1}{2}(1,2a+1,2,2a+1)$ & $2a\leq aw(X)$ \\ $u^2+y^2z+\lambda yx^{2a+1}+g(x,z)$ & $\frac{1}{2}(1,b',2,b')$ &  \\ 
			& $\frac{1}{2}(1,b'-1,2,b')$ & $b'\leq aw(X)$ \\ & $\frac{1}{2}(1,b'-1,2,b'+1)$ & \\ \hline
			$u^2\pm2uh(x,z)+y^2z$ & $\frac{1}{2}(1,b',2,b'+2)$ & $b'\leq aw(X)$ \\
			$+\lambda yx^{2a+1}+g(x,z)$ & &\\ \hline
			$u^2+y^2z\pm2yzh(x,z)$ & $\frac{1}{2}(1,b'+1,2,b'+1)$ & $b'\leq aw(X)$ \\ 
			$+\lambda yx^{2a+1}+g(x,z)$ & & \\ \hline
			$u^2+g(x,z)$ & $\frac{1}{2}(1,b'+1,2,b'+1)$ & $b'\leq aw(X)$ \\
			$+(y-h(x,z))(yz+h(x,z)z+\lambda x^{2a+1})$ & & \\\hline
			\end{tabular}\par
			One can take $d=6$ in first case. For the other three cases, note that by Lemma \ref{xidep}, $aw(X)\leq\Xi(X)\leq2dep(X)$. Hence we may take $d=4dep(X)+4$.
		\item $X$ is an LCI locus in $\A^5_{(x,y,z,u,t)}/\frac{1}{2}(1,1,0,1,1)$ and $Y$ is the weighted blow-up given by one of the following.\par
			\begin{tabular}{|l|c|c|} \hline
	 		Defining equations & Blowing-up weight & Relations \\ \hline
			$\left\lbrace\begin{array}{l}u^2+yt+x^{2a}+z^c\\ t-xz-y^3\end{array}\right.$ & $\frac{1}{2}(1,1,2,3,5)$ & \\ \hline
			$\left\lbrace\begin{array}{l}u^2+yt+g(x,z) \\ t-yz-\lambda x^{2a+1}\end{array}\right.$ & $\frac{1}{2}(1,2a-1,2,2a+1,2a+3)$ & $2a\leq aw(X)$ \\ \hline
			$\left\lbrace\begin{array}{l}u^2+zt+\lambda yx^{2a+1}+q(x) \\ t-y\pm2uh(x,z)+p(x,z)\end{array}\right.$ & $\frac{1}{2}(1,b',2,b'+1,2b'+2)$ & $b'\leq aw(X)$ \\ \hline
			$\left\lbrace\begin{array}{l}u^2+zt+\lambda yx^{2a-1}+q(x) \\ t-y^2+p(x,z)\end{array}\right.$ & $\frac{1}{2}(1,b'-1,2,b'+1,2b')$ & $b'\leq aw(X)$ \\ \hline
			$\left\lbrace\begin{array}{l}u^2+yt+g(x,z) \\ t-z(y+2h(x,z))+\lambda x^{2a+1}\end{array}\right.$ & $\frac{1}{2}(1,b'-1,2,b'+1,b'+3)$ & $b'\leq aw(X)$ \\ \hline
			\end{tabular}\par
			We take $d=6$ in first case and $d=4dep(X)+2$ in the remained four cases.
		\end{enumerate}
	\end{itemize}
\end{proof}
\begin{pro}\label{divX}
	Assume $f:Y\rightarrow X$ is divisorial contraction contract a divisor $E$ to a point $P\in X$. Then there is a number $D'_{dep(Y)}$ depends only on $dep(Y)$ such that
	\[|\ctop{Y}-\ctop{X}|\leq D'_{dep(Y)}.\]
\end{pro}
\begin{proof} We will discuss case by case.
\begin{enumerate}[(i)]
	\item $P$ is smooth. By \cite{Kaw1}, $Y\rightarrow X$ is a weighted blow-up, so $E$ is a 2 dimensional weighted projective space and then $|\ctop{Y}-\ctop{X}|\leq M^3_{1,1}+1$.
	\item $f$ is the usual blow-up. As Mori's classification locally $X$ can be viewed as a hypersurface in $\A^4$ and the exceptional divisor after blowing-up a point
		will be a degree two surface in	$\Pp^3$, hence $|\ctop{Y}-\ctop{X}|\leq M^3_{2,1}+1$.
	\item $P$ is of index $m>1$ and $a(X,E)=1/m$. By Proposition \ref{divm}, $|\ctop{Y}-\ctop{X}|\leq D_{dep(X)}$. Now Proposition \ref{depdiv} said that $dep(Y)\geq dep(X)-1$,
		hence \[|\ctop{Y}-\ctop{X}|\leq D_{dep(X)}\leq D_{dep(Y)+1}.\]
	\item $f$ is of ordinary type as in \cite{Kaw2}. By \cite{Kaw2}, Theorem 1.2, either the divisorial contraction belong to the previous cases,
		or one of the cases below. We will denote $m$ to be
		index of $P$, $a$ be the discrepancy $a(X,E)$ and $J=\{(r_1,1),(r_2,1)\}$, $\{(r,1),(r+2,1)\}$ or $\{(r,1),(r+4,1)\}$ be the non-Gorenstein data of $Y$ defined in \cite{Kaw2}.
		As the following table one can see that there is a upper bound of the weight of the exceptional divisor which can be written in terms of $dep(Y)$. Note that we have
		both $r_1+r_2$ and $2r+2\leq\Xi(Y)\leq2dep(Y)$.\par
		\begin{tabular}{|c|c|c|} \hline
		Defining equations & Blowing-up weight & Upper bound \\ \hline 
		$(xy+g(z^m,u)=0)$ & $(r_1/m,r_2/m,a,1)$ & $2dep(Y)$\\
			$\subset \A^4_{(x,y,z,u)}/\frac{1}{m}(1,-1,b,0)$ & &\\ \hline
		$(x^2+xq(z,u)+y^2u+\lambda yz^2+\mu z^3$ & $(r+1,r,a,1)$ & $2dep(Y)$ \\
		$+p(y,z,u)=0)\subset\A^4_{(x,y,z,u)}$ & &\\ \hline
		$(x^2+xzq(z^2,u)+y^2u+\lambda yz^{2\alpha-1}$ & $((r+2)/2,r/2,a,1)$ & $dep(Y)$ \\
		$+p(y,z,u)=0)\subset\A^4_{(x,y,z,u)}$ & &\\ \hline
		$\begin{pmatrix}yu+z^{(r+1)/a}+q(z,u)u+t=0 \\ x^2+xt+p(y,z,u)=0\end{pmatrix}$ & $(r+1,r,a,1,r+1)$ & $2dep(Y)$ \\
		$\subset\A^5_{(x,y,z,u,t)}$ & & \\ \hline
		$\begin{pmatrix}yu+z^{(r+2)/a}+q(z^2,u)zu+t=0 \\ x^2+xt+p(z^2,u)=0\end{pmatrix}$ & $((r+2)/2,r/2,a,1,(r+2)/2)$ & $dep(Y)+1$ \\
		$\subset\A^5_{(x,y,z,u,t)}/\frac{1}{2}(1,1,1,0,1)$ & &\\ \hline
		\end{tabular}		
	\item $f$ is of exceptional type as in \cite{Kaw2} and $a(X,E)=1$. In this case we use the results in \cite{Hay3}, \cite{Hay4} and \cite{Hay5}. We have several cases here.
		However the conclusion is, one can take the upper bound to be $\max\{4dep(Y)-6,2dep(Y)+1,30\}$.
		\begin{enumerate}
		\item $P$ is of type $cD$. By \cite{Hay4} Theorem 2.1-2.5, we have $X$ is a LCI locus in $\A^4$ or $\A^5$ and $Y$ is obtained by weighted blow-up with some fixed weight.
			All the possibility are listed below.\par
			\begin{tabular}{|c|c|c|c|} \hline
			Defining equations & Blowing-up weight & Relation & Upper bound \\ \hline 
			$x^2+y^2u+\lambda yz^l+g(z,u)$ & $(r,r-1,1,2)$ & $r+3\leq\Xi(Y)$ & $4dep(Y)-6$ \\ \hline
			$x^2+y^2u+2yuh(z,u)$ & $(r,r,1,1)$ & $2r\leq\Xi(Y)$ & $2dep(Y)$ \\ 
			$+\lambda yz^l+g(z,u)$ & & & \\ \hline
			$\left\lbrace\begin{array}{l}x^2+ut+\lambda yz^l+g(z,u) \\ y^2+2xh(z,u)+g'(z,u)-t\end{array}\right.$ & $(r+1,r,1,1,2r+1)$ & $2r+1\leq\Xi(Y)$ & $2dep(Y)+1$ \\ \hline
			$x^2+2xh(z,u)+y^2u$ & $(r+1,r,1,1)$ & $2r+1\leq\Xi(Y)$ & $2dep(Y)+1$ \\
			$+\lambda y z^l+g(z,u)$ & & & \\ \hline
			$\left\lbrace\begin{array}{l}x^2+yt+g(z,u) \\ yu+\lambda z^l+2uh(z,u)-t \end{array}\right.$ & $(r,r-1,1,1,r+1)$ & $2r\leq\Xi(Y)$ & $2dep(Y)$ \\ \hline
			\end{tabular}
		\item $P$ is of type $cE$. By \cite{Hay5} Theorem 1.1, we have $X$ is a LCI locus in $\A^4$ or $\A^5$, and $Y$ is a weighted blow-up as in the following table.\par
			\begin{tabular}{|c|c|}\hline
			Defining equations & Blowing-up weight \\ \hline
			$x^2+y^3+g(z,u)y+h(z,u)$ & $w=(2,2,1,1)$ \\ \hline
			$x^2+2xq(y,z)+y^3+g(z,u)y+h(z,u)$ & $w=(3,2,1,1)$\\ \hline
			$x^2+y^3+g(z,u)y+h(z,u)$ &  $w=(3,2,2,1)$ \\ \hline
			$x^2+y^3+3\xi y^2u^2+g(z,u)y+h(z,u)$  & $w=(4,3,2,1)$ \\ \hline
			$x^2\pm 2x(\alpha yu+\beta z^2+\gamma zu^2+\delta u^4)+y^3+g(z,u)y+h(z,u)$  & $w=(5,3,2,1)$\\ \hline
			$x^2+y^3+3(\lambda zu+\mu u^3)y^2+g(z,u)y+h(z,u)$  &$w=(5,4,2,1)$ \\ \hline
			$x^2+y^3+g(z,u)y+h(z,u)$  & $w=(6,4,3,1)$ \\ \hline
			$x^2+y^3+3(\lambda zu+\mu u^4)y^2+g(z,u)y+h(z,u)$  & $w=(7,5,3,1)$ \\ \hline
			$x^2\pm 2xu(\alpha yu+\beta z^2+\gamma zu^3+\delta u^6)+y^3+g(z,u)y+h(z,u)$  & $w=(8,5,3,1)$\\ \hline
			$x^2+y^3+g(z,u)y+h(z,u)$  & $w=(9,6,4,1)$ \\ \hline
			$x^2+y^3+3(\lambda zu^2+\mu u^6)y^2+g(z,u)y+h(z,u)$  & $w=(10,7,4,1)$ \\ \hline
			$x^2+y^3+g(z,u)y+h(z,u)$  &$w=(12,8,5,1)$ \\ \hline
			$x^2+y^3+g(z,u)y+h(z,u)$  & $w=(15,10,6,1)$ \\ \hline
			$\left\lbrace\begin{array}{l}x^2+y^3+l(z,u)t+g(z,u)y+h(z,u) \\ 2l'(z,u)x+g'(z,u)y+h'(z,u)-t\end{array}\right.$  &$w=(3,2,1,1,5)$ \\ \hline
			$x^2+2(l'(z,u)y+g'(z,u))x+y^3+g(z,u)y+h(z,u)$  & $w=(4,2,1,1)$ \\ \hline
			$\left\lbrace\begin{array}{l}x^2+y^3+q(z,u)t+g(z,u)y+h(z,u) \\ g'(z,u)y+h'(z,u)-t\end{array}\right.$  & $w=(3,2,1,1,4)$\\ \hline
			$x^2+y^3-3q(z,u)y^2+g(z,u)y+h(z,u)$  & $w=(3,3,1,1)$\\ \hline
			$\left\lbrace\begin{array}{l}x^2+yt+g(z,u)y+h(z,u) \\ y^2+3(\lambda zu+\mu u^3)y+g'(z,u)-t\end{array}\right.$  & $w=(5,3,2,1,7)$ \\ \hline
			\end{tabular}
			The weight of exceptional locus depends only on the type, one may take $d=30$ as an upper bound.
		\item $P$ is of type $cD/2$. By \cite{Hay3} Theorem 1.1, $X$ is an LCI locus in
			\[\A^4_{(x,y,z,u)}/\frac{1}{2}(1,1,1,0) \quad\mbox{or}\quad \A^5_{(x,y,z,u,t)}/\frac{1}{2}(1,1,1,0,0)\]
			and $Y$ is obtained by weighted blow-up as the following table.\par
			\begin{tabular}{|c|c|c|c|} \hline
			Defining equations & Blowing-up weight & Relation & Upper bound \\ \hline 
			$x^2+y^2u+s(z,u)yuz$ & $(2l,2l,1,1)$ & $8l\leq\Xi(Y)$ & $dep(Y)$ \\
			$+r(z)y+p(z,u)$ & & & \\ \hline
			$x^2+yzu+y^4+z^{2b}+u^c$ & $(2,2,1,1)$ & & $4$ \\ \hline
			$\left\lbrace\begin{array}{l}x^2+ut+r(z)y+p(z,u) \\ y^2+s(z,u)zx+q(z,u)-t\end{array}\right.$ & $(l+1,l,1,1,2l+1)$ & $4l+2\leq\Xi(Y)$ & $dep(Y)+1$ \\ \hline
			$\left\lbrace\begin{array}{l}x^2+yt+p(z,u) \\ yz+u^2-t\end{array}\right.$ & $(2l+2,2l,1,1,2l+2)$ & $8l+4\leq\Xi(Y)$ & $dep(Y)$ \\ \hline
			\end{tabular}			
		\item $P$ is of type $cE/2$. By \cite{Hay3} Theorem 1.2, $X$ is defined by \[(u^2+x^3+3\nu x^2z^2+p(x,y,z)=0)\subset\A^4_{(x,y,z,u)}/\frac{1}{2}(0,1,1,1)\] and $Y$ is obtained by
			weighted blow-up with weight $(3,2,1,4)$. The weight of exceptional locus is $6$.
		\end{enumerate}
	\item $f$ is of exceptional type as in \cite{Kaw2} and $a(X,E)=2$. We use the results in \cite{Kaw3}. We have that $X$ is isomorphic to
		\[(x^2+ut+p(y,z,u)=y^2+q(x,z,u)+t=0)\subset\A^5_{(x,y,z,u,t)}/\frac{1}{2}(1,1,1,0,0)\] and $Y$ is ontained by weighted blow-up of weight $((r+1)/2,(r-1)/2,2,1,r)$ and
		$2r\leq dep(Y)$. The weight of exceptional locus can be bounded by $dep(Y)$.
	\item $f$ is of exceptional type, $P$ is Gorenstein and discrepancy is greater than one. We use \cite{Yam}, Theorem 2.2-2.10.
		$X$ is a LCI locus in $\A^4$ or $\A^5$, $Y$ is weighted blow-up and all possible cases are listed below.\par
		\begin{tabular}{|c|c|c|c|} \hline
			Defining equations & Blowing-up weight & Relation & Upper bound \\ \hline 
			$\left\lbrace\begin{array}{l}x^2+\lambda yz^k+ut+p(z,u) \\ y^2+2xq(z,u)+r(z,u)+t\end{array}\right.$ & $(\frac{r+1}{2},\frac{r-1}{2},4,1,r)$ & $r\leq\Xi(Y)$ & $2dep(Y)$ \\ \hline 
			$\left\lbrace\begin{array}{l}x^2+\lambda yz^k+ut+p(z,u)\\ y^2+2xq(z,u)+r(z,u)+t\end{array}\right.$ & $(\frac{r+1}{2},\frac{r-1}{2},2,1,r)$ & $r\leq\Xi(Y)$ & $2dep(Y)$ \\ \hline 
			$x^2+y^2u+2yup(z,u)$ & $(r,r,2,1)$ & $2r\leq\Xi(Y)$ & $2dep(Y)$ \\ 
			$+\lambda yz^k+q(z,u)$ & & & \\ \hline
			$x^2+y^2u+2yup(z,u)$ & $(3,3,1,2)$ & & $6$ \\
			$+\lambda yz^k+q(z,u)$ & & & \\ \hline
			$x^2+(y-p(z,u))^3$ & $(3,3,2,1)$ & & $6$ \\
			$+yg(z,u)+h(z,u)$ & & & \\ \hline
			$x^2+y^2+2cxy+2xp(z,u)$ & $(4,3,2,1)$ & & $6$ \\
			$+2cyq(z,u)+z^3+g(z,u)$ & & & \\ \hline
			$x^2+y^2u+2yup(z,u)$ & $(3,4,2,1)$ & & $6$ \\ 
			$+\lambda yz^k+q(z,u)$ & & & \\ \hline
			$\left\lbrace\begin{array}{l}x^2+xt+p(z,u)\\ y^2+q(z,u)+t\end{array}\right.$ & $(5,3,2,2,7)$ & & $10$ \\ \hline 
			$x^2+y^3+\lambda y^2u^2$ & $(7,5,3,2)$ & & $14$ \\
			$+yg(z,u)+h(z,u)$ & & & \\ \hline
		\end{tabular}
\end{enumerate}
\end{proof}

The following lemma treat the blow-up LCI curve case.

\begin{lem}\label{lci}
	Assume $f:Y\rightarrow X$ is blow up LCI curve $C$ on $X$, then \[\ctop{Y}-\ctop{X}=\ctop{C}.\]
\end{lem}
\begin{proof} At firsr one show that over any point of $C$, the fiber is a $\Pp^1$. To see that, assume $C$ is defined by the ideal $I$ and locally $I$ is generated by the $\alpha$ and $\beta$.
	Then  $Y$ is isomorphic to $Proj\bigoplus_{n\geq0}I^n$ and the natural map $\Oo_X[x,y]\rightarrow\bigoplus_{n\geq0}I^n$ defined by $x\mapsto\alpha$, $y\mapsto\beta$ gives an inclusion
	$Y\hookrightarrow X\times\Pp^1$. Hence all fiber along $C $ is a $\Pp^1$. Now there exists a open set $U\subset C$ such that $f^{-1}U\cong U\times\Pp^1$ since
	geometric ruled surface are ruled, hence if one denote $E$ to be the exceptional divisor of $f$, then \[\ctop{E}=\ctop{f^{-1}U}+\ctop{f^{-1}(C-U)}=2\ctop{U}+2\ctop{C-U}=2\ctop{C}\] and then
	\[\ctop{Y}-\ctop{X}=\ctop{E}-\ctop{C}=\ctop{C}.\]
\end{proof}
 
 Now let $X$ be a smooth threefold and consider the process of minimal model program $X=X_0\dashrightarrow X_1\dashrightarrow...\dashrightarrow X_m=X_{min}$.
 We will use above results to estimate the third Betti number of $X_i$.

\begin{pro}\label{b3}
	Let $X\rightarrow W$ be a divisorial contraction and $X\dashrightarrow X'$ be flip or flop. Then there is a constant $\Phi_{dep(X)}$ depends only on $dep(X)$ such that
	$b_3(W)\leq\Phi_{dep(X)}+b_3(X)$ and $b_3(X')\leq\Phi_{dep(X)}+b_3(X)$
\end{pro}
\begin{proof}
	Assume $X\rightarrow W$ is a divisorial contraction to point, then by Corollary \ref{ctop} and Proposition \ref{divX} we have $|b_3(X)-b_3(W)|=|\ctop{X}-\ctop{W}-2|\leq D'_{dep(X)}+2$,
	hence \[b_3(W)\leq D'_{dep(X)}+2+b_3(X).\] If $X\rightarrow W$ is blow-up LCI curve on $W$, then using Corollary \ref{ctop} and Lemma \ref{lci} one has
	$b_3(W)-b_3(X)=\ctop{X}-\ctop{W}-2=\ctop{C}-2\leq0$, hence \[b_3(W)\leq b_3(X).\] If $X\dashrightarrow X'$ is a flop, then $b_3(X)=b_3(X')$. So we only have
	to consider the cases when $X\rightarrow W$ is a divisorial contraction to curve which is not blow-up LCI curve and $X\dashrightarrow X'$ is a flip.\par
	Note that when $X$ is Gorenstein, there is no flipping contraction and every divisorial contraction to curve is blowing-up LCI curve, as mentioned in
	Remark \ref{Gcurve}. Hence we may induction on the number $dep(X)$. By Theorem \ref{ch} we have the diagram
	\[\xymatrix{ Y \ar@{-->}[rr]\ar[d]^f & & Y' \ar[d]_{f'} \\ X\ar[rd]^g && X'\ar[ld]_{g'}	\\ & W &}.\]
	At first note that by Proposition \ref{BettiD} and Proposition \ref{divm}, \[|b_3(Y)-b_3(X)|=|\ctop{Y}-\ctop{X}|\leq D_{dep(X)},\] hence $b_3(Y)\leq D_{dep(X)}+b_3(X)$. On the other hand,
	one write \[Y=Y_0\dashrightarrow Y_1\dashrightarrow...\dashrightarrow Y_l=Y',\] by Remark \ref{rch} we have $Y_i\dashrightarrow Y_{i+1}$ is a flip
	for $i>0$, hence $dep(Y_{i+1})<dep(Y_i)$ for all $i>0$ and $dep(Y_0)<dep(X)$. Thus $l\leq dep(X)$. By induction hypothesis we have $b_3(Y_{i+1})<\Phi_{dep(Y_i)}+b_3(Y_i)$.
	Now define $\Psi^0_{dep(X)}=D_{dep(X)}$ and $\Psi^n_{dep(X)}=\Phi_{dep(X)-1}+\Psi^{n-1}_{dep(X)}$, then we have
	\[b_3(Y_0)=b_3(Y)\leq\Psi^0_{dep(X)}+b_3(X)\] and hence \[b_3(Y_{i+1})\leq\Phi_{dep(Y_i)}+b_3(Y_i)\leq\Phi_{dep(X)-1}+b_3(Y_i)=\Psi^{i+1}_{dep(X)}+b_3(X)\]
	by induction on $i$. We conclude that $b_3(Y')=b_3(Y_l)\leq\Psi^{dep(X)}_{dep(X)}+b_3(X)$. Finally
	\[b_3(X')\leq\Phi_{dep(Y')}+b_3(Y')\leq\Phi_{dep(X)-1}+\Psi^{dep(X)}_{dep(X)}+b_3(X)\] since $dep(Y')<dep(X)$. So we finish the case when $X\dashrightarrow X'$ is a flip.\par
	Now assume $X\rightarrow W$ is divisorial contraction to curve, then one has to estimate $b_3(W)$. In this case $g':X'\rightarrow W$ is divisorial contraction to point,
	hence one may apply Proposition \ref{divX} to get $|b_3(X')-b_3(W)|=|\ctop{X'}-\ctop{W}|\leq D'_{dep(X')}$ and then
	\[b_3(W)\leq D'_{dep(X')}+b_3(X')\leq D'_{dep(X)}+\Phi_{dep(X)-1}+\Psi^{dep(X)}_{dep(X)}+b_3(X).\]
	
\end{proof}

\begin{proof}[Proof of Theorem \ref{mthm}]
	$(i)$ and $(ii)$ are Proposition \ref{b2}. Also as in Remark \ref{deprho} we have $dep(X_i)\leq\rho(X)$ for all $i$. So Proposition \ref{b3} implies
	\[b_3(X_i)\leq\Phi_{\rho(X)}+b_3(X_{i-1})\leq i\Phi_{\rho(X)}+b_3(X).\] Now $i\leq2\rho(X)$ by \cite{CZ}, Lemma 3.1. One conclude that one can take
	$\bar{\Phi}_{\rho(X)}=2\rho(X)\Phi_{\rho(X)}$.
\end{proof}

\section{Examples and applications}
Let $Y\rightarrow X$ be a extremal divisorial contraction between terminal threefolds, then as Lemma \ref{BettiD} $b_i(Y)-b_i(X)$ are known except for $b_3$. In the previous section
we have shown that $|b_3(Y)-b_3(X)|$ can be bounded by some constant depend only on the depth of $X$ or $Y$. The following examples shows that the bound is truly depends on the
depth. When the depth being larger, the bound should be larger.
\begin{eg}
	Assume $P\in X$ is locally isomorphic to the origin in \[(x^2+y^2+z^{4k+2}+u^{2k+1}=0)\subset\A^4_{(x,y,z,u)}/\frac{1}{4}(1,3,1,2).\] This is a isolated terminal point of type $cAx/4$.
	Assume $k$ is even and let $Y$ be the weighted blow up of weight $\frac{1}{4}(2k+1,2k+3,1,2)$. Then $Y\rightarrow X$ is a extremal divisorial contraction with discrepancy $1/4$.
	Let $E$ be the exceptional divisor. We have \[b_3(X)-b_3(Y)=\ctop{Y}-\ctop{X}-2=\ctop{E}-3.\] Hence to compute $b_3(X)-b_3(Y)$ is equivalent to compute $\ctop{E}$.\par
	Now in this case \[E\cong(x^2+z^{4k+2}+u^{2k+1}=0)\subset\Pp(2k+1,2k+3,1,2).\] On $U_z=\{z=1\}$ we have $E|_{U_z}\cong(x^2+u^{2k+1}+1)\subset\A^3_{(x,y,u)}$.
	This is a line bundle over a smooth curve $C=(x^2+u^{2k+1}+1)\subset\A^2_{(x,u)}$ which is of degree $2k+1$, hence \[\ctop{E|_{U_z}}=\ctop{C}=-(2k-2)(2k+1)-(2k+1),\]
	which tends to $-\infty$ as	$k$ tends to $\infty$.\par
	On the other hand, one can show that $E|_{\{z=0\}}$ is isomorphic to $\Pp^1$. Hence $\ctop{E}$ tends to $-\infty$ when $k$ tends to $\infty$. This shows that $b_3(X)-b_3(Y)$ could be
	arbitrary negative.
\end{eg}
\begin{eg}
	Assume $P\in X$ is locally isomorphic to the origin in \[(xy+z^{mk}+u^k=0)\subset\A^4_{(x,y,z,u)}/\frac{1}{m}(\alpha,-\alpha,1,m)\] with $(\alpha,m)=1$.
	This is a isolated terminal point of type $cA/m$.
	Let $Y$ be the weighted blow up of weight $\frac{1}{m}(a,b,1,m)$ with $a\equiv\alpha$ mod $m$ and $a+b=mk$. Then $Y\rightarrow X$ is a extremal divisorial contraction with discrepancy $1/m$.
	The exceptional divisor $E$ is isomorphic to \[(xy+z^{mk}+u^k=0)\subset\Pp(a,b,1,m).\] On the affine open set $U_y=\{y=1\}$ we have
	$E|_{U_y}\cong(x+z^{mk}+u^k=0)\subset\A^3/\frac{1}{b}(a,1,m)$, which is isomorphic to $\A^2/\frac{1}{b}(1,m)$. One can compute that $\ctop{E|_{U_y}}=1$.\par
	Now let \[E'=E|_{\{y=0\}}\cong(z^{mk}+u^k=0)\subset\Pp(a,1,m).\] We have $E'|_{\{z=1\}}\cong(u^k+1=0)\subset\A^2_{(x,u)}$, which is $k$ lines. Also $E'|_{\{z=0\}}$ is a point,
	hence \[\ctop{E'}=k+1.\]\par
	A conclusion is that $\ctop{E}=k+2$ can be arbitrary large when $k$ growth to infinity, hence $b_3(X)-b_3(Y)$ could be arbitrary positive.
\end{eg}
In the rest part we will prove Theorem \ref{iBetti}. From now on let $X$ be a projective $\Q$-factorial terminal threefold over $\Cc$. For any singular point $P\in X$,
we say that there exists a \emph{feasible resolution} for $P$ if there is a sequence
\[X_n\rightarrow X_{n-1}\rightarrow...\rightarrow X_0=X\]
so that $X_n$ is smooth over $P$ and $X_i\rightarrow X_{i+1}$ is a extremal divisorial contraction to point with minimal discrepancy. 
\begin{thm}[\cite{C}, Theorem 2]
	Given a three-dimensional terminal singularity $P\in X$, there exists a feasible resolution for $P\in X$.
\end{thm}
\begin{cor}
	Given a projective $\Q$-factorial terminal threefold $X$ over $\Cc$, there is a smooth variety $Y$ such that $Y\rightarrow X$
	is a composition of steps of $K_Y$-minimal model program, and the relatively Picard number $\rho(Y/X)$ depends only on the singularity (that is, the local equation near singular points) of $X$.
\end{cor}
\begin{cor}\label{the}
	Notation as above. we have $b_i(Y)\leq b_i(X)+\Theta_i$, where $\Theta_i$ is a constant depends only on singularities of $X$ and $\rho(X)$.
\end{cor}
\begin{proof}
	We apply Theorem \ref{mthm}.
	When $i=0,1,5,6$, one take $\Theta_i=0$. For $i=2,4$ we choose $\Theta_i$ to be $\rho(Y/X)$. Now assume $i=3$ and assume $Y\rightarrow X$
	factors through \[Y=X_n\rightarrow X_{n-1}\rightarrow...\rightarrow X_0=X\]
	with $X_i\rightarrow X_{i+1}$ is a extremal divisorial contraction to point. By Proposition \ref{divX} and Corollary \ref{ctop} we have
	\[|b_3(X_{i+1})-b_3(X_i)|\leq |\ctop{X_i}-\ctop{X_{i+1}}|+2\leq D'_{dep(X_{i+1})}+2.\]
	Now $n$ is equal to $\rho(Y/X)$ and $dep(X_{i+1})$ is bounded by $\rho(Y)=\rho(Y/X)+\rho(X)$ (Remark \ref{deprho}). Hence 
	\[|b_3(Y)-b_3(X)|\leq n(D'_{\rho(Y)}+2)\] is a constant depends on singularities of $X$ and $\rho(X)$.
\end{proof}
\begin{proof}[Proof of Theorem \ref{iBetti}]
	Let \[ Y=X_n\rightarrow X_{n-1}\rightarrow...\rightarrow X_0=X \] be a feasible resolution.
	By \cite{CT} Lemma 2.16 we have
	\[0\rightarrow IH^i(X_j,\Q)\rightarrow IH^i(X_{j+1},\Q)\oplus IH^i(P_j,\Q)\rightarrow IH^i(E_j,\Q)\rightarrow0\]
	is exact for $i\geq1$, here $E_j=exc(X_{j+1}\rightarrow X_j)$ and $P_j$ is the image of $E_j$. Hence $Ib_i(X_{j+1})\geq Ib_i(X_j)$ for all $j$.
	Thus $Ib_i(X)\leq Ib_i(Y)=b_i(Y)\leq b_i(X)+\Theta_i$ by Corollary \ref{the}.
\end{proof}

\end{document}